\documentclass[12pt]{amsart}
\usepackage{amssymb, amscd}

\usepackage{mathptmx}

\usepackage[latin1]{inputenc}

\newcommand{\R}{\mathbb{R}}
\newcommand{\N}{\mathbb{N}}

\newcommand{\Z}{\mathbb{Z}}

\newcommand{\Ec}{\mathcal{E}}
\newcommand{\Fc}{\mathcal{F}}
\newcommand{\Gc}{\mathcal{G}}

\newcommand{\Wc}{\mathcal{W}}
\DeclareMathOperator{\pnt}{\raise 0.5mm \hbox{\large\bf.}}

\DeclareMathOperator{\Tor}{Tor}
\DeclareMathOperator{\conv}{conv}
\DeclareMathOperator{\cone}{cone}

\DeclareMathOperator{\GL}{GL}
\DeclareMathOperator{\inom}{in}
\DeclareMathOperator{\gT}{gT}

\DeclareMathOperator{\gin}{gin}

\DeclareMathOperator{\wt}{wt}

\DeclareMathOperator{\GF}{GF}
\DeclareMathOperator{\gGF}{gGF}

\DeclareMathOperator{\aff}{aff}

\DeclareMathOperator{\depth}{depth}
\DeclareMathOperator{\gdepth}{gdepth}
\DeclareMathOperator{\projdim}{projdim}
\DeclareMathOperator{\charac}{char}

\newtheorem{thm}{\bf Theorem} [section]
\newtheorem{lemma}[thm]{\bf Lemma}
\newtheorem{cor}[thm]{\bf Corollary}
\newtheorem{prop}[thm]{\bf Proposition}

\theoremstyle{definition}
\newtheorem{defn}[thm]{\bf Definition}

\newtheorem{rem}[thm]{\bf Remark}
\newtheorem{ex}[thm]{\bf Example}

\theoremstyle{plain}
\newtheorem*{thm*}{Theorem}

\title{Algebraic Properties of Generic Tropical Varieties}

\author{Tim R\"omer}
\address{Universit\"at Osnabr\"uck, Institut f\"ur Mathematik, 49069 Osnabr\"uck, Germany}
\email{troemer@uos.de}

\author{Kirsten Schmitz}
\address{Universit\"at Osnabr\"uck, Institut f\"ur Mathematik, 49069 Osnabr\"uck, Germany}
\email{kischmit@uos.de}

\begin{document}

\begin{abstract}
We show that the algebraic invariants multiplicity and depth of a graded ideal in the polynomial ring are closely connected to the fan structure of its generic tropical variety in the constant coefficient case. Generically the multiplicity of the ideal is shown to correspond directly to a natural definition of multiplicity of cones of tropical varieties. Moreover, we can recover information on the depth of the ideal from the fan structure of the generic tropical variety if the depth is known to be greater than $0$. In particular, in this case we can see if the ideal is Cohen-Macaulay or almost-Cohen-Macaulay from its generic tropical variety.
\end{abstract}

\maketitle

\section{Introduction}

As a very new area of mathematics tropical geometry has received a lot of attention from various points of view in the last few years; see \cite{DEST,KAMAMA,MI2,SPST} and \cite{GA,ITMISH} for review articles. One approach to tropical geometry is to associate a combinatorial object to a projective algebraic variety which provides an effective tool for studying questions in algebraic geometry, see for example \cite{DR, GAMA2}. More precisely, the tropical variety $T(X)$ of an algebraic variety $X$ is the real-valued image of $X$ under some valuation map, see \cite{DR,JEMAMA,SPST}. In certain settings $T(X)$ has the structure of a polyhedral complex (see \cite{BIGR,JE}) and there is a practical characterization in terms of initial ideals given in \cite{SPST} and \cite[Theorem 4.2]{DR}. If the valuation on the ground field is trivial, $T(X)$ is a subfan of the Gröbner fan of the ideal $I$ defining $X$. We will only consider this constant coefficient case and we will define the tropical variety as a fan associated to $I$ instead of $X$. In this situation the ideal $I$ need not be a radical ideal. So for our purposes let $K$ be an infinite field and $K[x_1,\ldots,x_n]$ be the polynomial ring in $n$ variables over $K$. In this setting the tropical variety $T(I)$ of a graded ideal $I\subset K[x_1,\ldots,x_n]$ is defined to be the subfan of the Gröbner fan of $I$ which consists of all cones such that the corresponding initial ideal does not contain a monomial.

The tropical variety of an ideal depends on the choice of coordinates in the following sense. For $g\in \GL_n(K)$ the image $g(I)$ of a graded ideal $I\subset K[x_1,\ldots,x_n]$ is also a graded ideal and all important algebraic invariants of $I$ such as the dimension, multiplicity or depth are preserved under $g$. In fact, $g(I)$ can be considered as the ideal $I$ given in different coordinates. In general, for $I\subset K[x_1,\ldots,x_n]$ and $g\in \GL_n(K)$ we have $T(I)\neq T(g(I))$. We can, however, find a non-empty Zariski-open set $U\subset \GL_n(K)$, such that $T(g(I))$ is the same fan $\gT(I)$ for every $g\in U$; see \cite[Corollary 6.9]{TK}. The fact that $U$ is dense in $\GL_n(K)$ justifies the name generic tropical variety of $I$ for this fan. In \cite[Corollary 8.4]{TK} it was shown that $\gT(I)$ as a set depends only on the dimension of $I$. More precisely, for an $m$-dimensional ideal the underlying set of $\gT(I)$ is always the $m$-skeleton of a particular complete fan $\Wc_n$ in $\R^n$. In this paper we will show that under certain conditions we can recover information on the depth of $I$ in addition to the dimension from the fan structure of $\gT(I)$ induced by the Gröbner fan. As one of the main results we can completely describe generic tropical varieties of Cohen-Macaulay and almost-Cohen-Macaulay ideals as fans. With this we can determine if an ideal is Cohen-Macaulay or almost-Cohen-Macaulay from the fan structure of its generic tropical variety if we know the depth to be greater than $0$. Moreover, we show that the multiplicities associated to the maximal cones of $\gT(I)$ as done in \cite{DIFEST} correspond directly to the multiplicity of $I$.

Our paper is organized as follows. In Section 2 we will introduce the basic results and the notation needed in the sequel. In Section 3 we show that for an $m$-dimensional ideal the generic tropical variety is always a subfan of the $m$-skeleton of $\Wc_n$ by showing that the fan structure induced by $\Wc_n$ is the coarsest possible on the underlying set. This will be important for all following sections. The next two Sections 4 and 5 are devoted to the depth of $I$. In Section 4 we show that $\gT(I)$ is equal to the $m$-skeleton of $\Wc_n$ if and only if $I$ is Cohen-Macaulay or almost-Cohen-Macaulay with $\dim I=m$. In Section 5 we show that we can recover the depth of $I$ from $\gT(I)$ if we know it to be greater than $0$ and less than $\dim I-1$. We also give more structural results depending on $\depth(I)$ on $\gT(I)$ as a fan for a special class of ideals. We then show that the multiplicities defined on the maximal cones of $T(I)$ as in \cite{DIFEST} generically behave in a nice way in Section 6. These multiplicities coincide with the multiplicity of $I$.

We thank Hannah Markwig and Bernd Sturmfels for useful suggestions for this paper and we are especially grateful to Diane Maclagan for many illuminating discussions.

\section{Preliminaries}

In the following let $K$ be an algebraically closed field of characteristic $0$ and $K[x_1,\ldots,x_n]$ be the polynomial ring in $n$ variables over $K$. The \emph{$\omega$-weight} $\wt_{\omega}(cx^{\nu})$ of some term $cx^{\nu}=cx_1^{\nu_1}\cdots x_n^{\nu_n}\in K[x_1,\ldots,x_n]$ is defined as $\wt_{\omega}(cx^{\nu})=\omega\cdot \nu$ for any $\omega\in\R^n$. For a homogeneous polynomial $f\in K[x_1,\ldots,x_n]$ with $f=\sum_{\nu\in\N^n} a_{\nu}x^{\nu}$ and $\omega\in\R^n$ the \emph{initial polynomial} $\inom_{\omega}(f)$ of $f$ consists of all terms of $f$ such that their $\omega$-weight $\omega\cdot \nu$ is minimal. We will use multiplicative term orders $\succ$ on the monomials of $K[x_1,\ldots,x_n]$ and define $\inom_{\succ}(f)$ to be the term $cx^{\nu}$ of $f$ for which $cx^{\nu}\succ dx^{\mu}$ for every other term $dx^{\mu}$ of $f$. For $\omega\in\R^n$ and a term order $\succ$ we can consider the refinement $\succ_{\omega}$. This is the term order which first compares terms by their $\omega$-weight and uses $\succ$ to break ties. Note that while initial polynomials with respect to $\omega$ are defined by taking terms of minimal $\omega$-weight, the symbol $\succ$ suggests that $\inom_{\succ}(f)$ is the ''largest'' term of $f$. The reason for considering this counterintuitive setup is that in Gröbner basis theory one usually considers the largest terms as initial terms, while in tropical geometry it is convenient to work with the minimal $\omega$-weight.

We will consider graded ideals $I\subset K[x_1,\ldots,x_n]$ and always assume $I\neq(0)$ if not stated otherwise. The dimension $\dim I$ of $I$ refers to the Krull dimension of the coordinate ring $K[x_1,\ldots,x_n]/I$. Since we assume $I\neq(0)$, we always have $\dim I<n$. The \emph{initial ideal of $I\subset K[x_1,\ldots,x_n]$ with respect to $\omega\in \R^n$} is defined as $\inom_{\omega}(I)=(\inom_{\omega}(f):f\in I)$.

For $I\subset K[x_1,\ldots,x_n]$ we define the \emph{tropical variety} of $I$ by $$T(I)=\left\{\omega\in\R^n:\inom_{\omega}(I) \text{ does not contain a monomial }\right\}.$$ This is a special case, called the \emph{constant coefficient case}, of the usual definition of a tropical variety as the image of a projective variety under a valuation map, see for example \cite{DR,SPST}. In this case $K$ is considered to have a trivial valuation, see \cite[Theorem 4.2]{DR}. Then the tropical variety $T(I)$ is a subfan of the \emph{Gröbner fan $\GF(I)$} of $I$ as was observed in \cite{BJSST}. Recall that the Gröbner fan is a complete fan in $\R^n$, where $\omega,\omega'\in\R^n$ are in the same relatively open cone if $\inom_{\omega}(I)=\inom_{\omega'}(I)$; see for example \cite{MORO} or \cite{ST}. Sometimes we denote the ideal $\inom_{\omega}(I)$ for a relatively open cone $\mathring{C}$  of $\GF(I)$ and $\omega\in\mathring{C}$ by $\inom_C(I)$.

We study the structure of the tropical variety under a generic coordinate transformation in the following sense. For $g\in\GL_n(K)$ we regard the $K$-algebra automorphism induced by
\begin{eqnarray*}
K[x_1,\ldots,x_n] & \longrightarrow & K[x_1,\ldots,x_n]\\
x_i                 & \longmapsto     & \sum_{j=1}^n g_{ji}x_j.
\end{eqnarray*}
In the sequel we identify $g$ with this automorphism and call both of them $g$. Note that this definition differs from \cite[Definition 2.5]{TK} by a transposition of the matrix $g$. However, this does not affect the results proved in \cite{TK}. We consider $\GL_n(K)$ equipped with the Zariski-topology. If $I$ is $0$-dimensional, for every $g\in\GL_n(K)$ the tropical variety $T(g(I))=\emptyset$, see \cite[Lemma 2.6]{TK}. We will therefore always assume that $\dim I>0$. In \cite[Corollary 6.9]{TK} it was shown that for a graded ideal $I\subset K[x_1,\ldots,x_n]$ with $\dim I>0$ there exists a Zariski-open set $\emptyset\neq U\subset\GL_n(K)$ such that $T(g(I))$ is the same fan for every $g\in U$. This fan is denoted by $\gT(I)$ and called the \emph{generic tropical variety} of $I$. If $g\in U$, then $g(I)$ is called a \emph{generic coordinate transformation} of $I$. Moreover, by \cite[Theorem 3.1]{TK} we know that there is also a \emph{generic Gröbner fan $\gGF(I)$} such that $\GF(g(I))=\gGF(I)$ as a fan for every $g\in U$. The monomial initial ideal $\inom_{\succ}(g(I))$ with respect to a term order $\succ$ is exactly the \emph{generic initial ideal $\gin_{\succ}(I)$} for $g\in U$. These generic initial ideals correspond to the maximal cones of $\gGF(I)$. In the following we will fix a nonempty Zariski-open subset $U\subset\GL_n(K)$ such that $\GF(g(I))=\gGF(I)$ and $T(g(I))=\gT(I)$ for every $g\in U$ and refer to it simply as $U$.

The generic tropical variety as a set is always equal to some skeleton of a particular complete fan $\Wc_n$ in $\R^n$. We recall that this fan is defined by the maximal cones $C_i=\left\{\omega\in\R^n: \omega_i=\min_k\left\{\omega_k\right\}\right\}$ for $i=1,\ldots,n$. Note that to define a fan in $\R^n$ or to show that two fans in $\R^n$ are the same it suffices to do this for the maximal cones. This follows from the fact that every cone in a fan is a face of a maximal cone, so all cones in a fan are determined by the maximal cones. Every $m$-dimensional cone $C_A$ in $\Wc_n$ for $m\in\left\{1,\ldots,n\right\}$ has the form $C_A=\left\{\omega\in\R^n: \omega_i=\min_k\left\{\omega_k\right\} \text{ for } i\in A\right\}$, where we have $A\subset \left\{1,\ldots,n\right\}$ with $\left|A\right|=n-m+1$. On the other hand every set $\emptyset\neq A\subset \left\{1,\ldots,n\right\}$ defines a cone of $\Wc_n$ in this way which we will denote by $C_A$. We let $\Wc_n^m$ be the $m$-skeleton of $\Wc_n$, that is the fan consisting of all cones of $\Wc_n$ of dimension less than or equal to $m$. In \cite[Corollary 8.4]{TK} it was shown that for a graded ideal $I\subset K[x_1,\ldots,x_n]$ with $\dim I=m$ the generic tropical variety $\gT(I)$ coincides with $\Wc_n^m$ as a set.

For a fan $\Fc$ in $\R^n$ we denote by $\left|\Fc\right|$ the set $\bigcup_{C\in \Fc} C$ (without its fan structure), where the union is taken over all cones of $\Fc$. The notation $C$ for a cone of $\Fc$ always refers to a closed cone. By $\mathring{C}$ we denote the relative interior of $C$. We say that a fan $\Ec$ in $\R^n$ \emph{refines} a fan $\Fc$ in $\R^n$, if for every relatively open cone $\mathring{C}$ of $\Ec$ there exists a relatively open cone $\mathring{D}$ of $\Fc$ with $\mathring{C}\subset \mathring{D}$. In Proposition \ref{sanogo} we show that it suffices to check this condition for the maximal cones of $\Ec$.

\section{Fan Structures on the set $\left|\Wc_n^m\right|$}

In this section we will always assume $0<m<n$. The aim is to show that $\Wc_n^m$ is the coarsest fan structure on the set $\left|\Wc_n^m\right|$. By this we mean that every fan $\Fc$ in $\R^n$ with $\left|\Fc\right|=\left|\Wc_n^m\right|$ refines $\Wc_n^m$ as a fan. For this we first prove that any fan on $\left|\Wc_n^m\right|$ is pure by proving this statement for any subset of $\R^n$ which permits a pure fan structure of dimension at most $n-1$. We repeatedly need the following Lemma.

\begin{lemma}\label{allofs}
Let $\Fc$ be a fan in $\R^n$ and $C$ a cone of $\Fc$. Let $\omega\in \mathring{C}$ and $(\omega_i)_{i\in\N}$ be a sequence such that $\omega_i\in\left|\Fc\right|\backslash C$ and $\lim_{i\rightarrow\infty}\omega_i=\omega$. Then there exists a cone $D$ in $\Fc$ containing a subsequence of $(\omega_i)_{i\in\N}$ such that $C$ is a proper face of $D$.
\end{lemma}

\begin{proof}
Since $\omega_i\in \Fc$, there exists some other cone $C_i\neq C$ such that $\omega_i\in \mathring{C}_i$. But $\Fc$ has only finitely many cones, so there exists a subsequence $(\omega_{j_i})_{j_i\in \N}$ of $(\omega_i)_{i\in\N}$ such that $\omega_{j_i}\in D$ for one particular cone $D$ of $\Fc$. By the choice of $\omega_i$ we have $D\neq C$. Now $\lim_{i\rightarrow \infty} \omega_{j_i}=\omega$ and $D$ is closed, so $\omega\in D$. Because $\mathring{C}\cap \mathring{D}=\emptyset$, we have $\omega\in\partial D$. By assumption $C$ and $D$ intersect in a face of both of them. Since $\omega\in \mathring{C}$ is in this intersection, this face is $C$. Hence, $C\subsetneq D$ as a face.
\end{proof}

With this we can show in the following proposition that any fan structure on $\left|\Wc_n^m\right|$ is pure.

\begin{prop}\label{baumann}
Let $\Ec$ be a pure $m$-dimensional fan in $\R^n$ and $\Fc$ an arbitrary fan in $\R^n$ with $\left|\Fc\right|=\left|\Ec\right|$. Then $\Fc$ is also a pure $m$-dimensional fan.
\end{prop}

\begin{proof}
Let $C$ be any cone in $\Fc$. Assume that $\dim C<m$ and let $\omega\in \mathring{C}$. Since for any open neighborhood $U(\omega)\subset \R^n$ of $\omega$ we have $\dim U(\omega)\cap C<m$, there always exists $v\in (U(\omega) \cap \left|\Fc\right|)\backslash C$. So if we choose a sequence $(\epsilon_n)_{n\in\N}$ with $\epsilon_n>0$ for every $n\in\N$ and $\lim_{n\rightarrow \infty} \epsilon_n=0$, there exists $v_n\in \left|\Fc\right|\backslash C$ with $\left|v_n-\omega\right|<\epsilon_n$. By Lemma \ref{allofs} we obtain a cone $D$ of $\Fc$ such that $C\subsetneq D$. Since $\dim D>\dim C$, either the proof is complete if $\dim D=m$, or we can apply the same procedure to $D$ instead of $C$. Either way we obtain an $m$-dimensional cone of which $C$ is a face after finitely many steps.
\end{proof}

Note that this immediately implies the following corollary which is a generalization of the fact that the tropical variety of prime ideals of dimension $m$ is a pure $m$-dimensional fan, see \cite{BIGR}.

\begin{cor}
Let $I\subset K[x_1,\ldots,x_n]$ be a graded $m$-dimensional ideal. Then $\gT(I)$ is a pure $m$-dimensional fan.
\end{cor}

To prove that a fan $\Ec\subset \R^n$ refines another fan $\Fc\subset \R^n$ it suffices to consider the maximal cones of $\Ec$. This will be the result of the next two statements.

\begin{lemma}\label{pasanen}
Let $D,C$ be cones in $\R^n$ such that $D\subset C$ and $D\cap \mathring{C}\neq\emptyset$. Then $\mathring{D}\subset\mathring{C}$.
\end{lemma}

\begin{proof}
Let $p\in \mathring{D}$ and for some $\epsilon>0$ let $U=\left\{u\in \mathring{D}: \left|u-p\right|<\epsilon\right\}\subset \mathring{D}$ be a relatively open neighborhood of $p$ in $\mathring{D}$. If $p\in \partial C$, then there exists a face $F$ of $C$ with $p\in F$. Let $H=\left\{\omega\in\R^n: a\cdot\omega=0\right\}$ be a defining hyperplane of $F$, so $F=H\cap C$ and let $C\subset H^-=\left\{\omega\in\R^n: a\cdot\omega\leq 0\right\}$. Since $U\subset \mathring{D}\subset C$, we know that $a\cdot u\leq0$ for every $u\in U$. In addition we have $a\cdot p=0$, because $p\in F$. Assume there exists $u\in U$ such that $a\cdot u<0$. Then we can choose $0<\lambda<1$ very small such that $p+\lambda(p-u)\in \mathring{D}$. Moreover, $\left|(p+\lambda(p-u))-p\right|=\left|\lambda(p-u)\right|<\epsilon$, so $p+\lambda(p-u)\in U$. But $(p+\lambda(p-u))\cdot a=-\lambda u\cdot a>0$ which is a contradiction to $p+\lambda(p-u)\in C$. Hence, $a\cdot u=0$ for every $u\in U$. Thus $U\subset H$ and since $U$ is relatively open in $D$, we also have $\aff(D)\subset H$. But then $$D\subset \aff(D)\cap C\subset H\cap C=F$$ which is a contradiction to $D\cap \mathring{C}\neq\emptyset$. Hence, $p\notin \partial C$ and we get that $\mathring{D}\subset \mathring{C}$.
\end{proof}

\begin{prop}\label{sanogo}
Let $\Ec,\Fc\subset \R^n$ be two fans. Then $\Ec$ refines $\Fc$ as a fan if and only if for every maximal cone $C\subset \Ec$ there exists a cone $D\subset \Fc$ such that $\mathring{C}\subset\mathring{D}$.
\end{prop}

\begin{proof}
We have to show that for any cone $K\subset \Ec$ there exists a cone $L\subset \Fc$ such that $\mathring{K}\subset\mathring{L}$. If $K$ is maximal, this is true by assumption. Let $K\subset \Ec$ be not maximal. Then there exists a maximal cone $C\subset \Ec$ such that $K$ is a face of $C$. Moreover, we know that $\mathring{C}\subset \mathring{D}$ for some cone $D\in \Fc$. So $K\subset D$. Assume that such a cone $L$ does not exist. If $K\cap \mathring{D}\neq\emptyset$, this would imply $\mathring{K}\subset\mathring{D}$ by Lemma \ref{pasanen} and we could set $L=D$. Hence, $K\cap \mathring{D}=\emptyset$. Then $K\subset \partial D$ and by \cite[Lemma 1.5]{BRGU} it follows that $K\subset E$ for a proper face $E$ of $D$. Since $\dim E<\dim D$, we can use a suitable induction to obtain a sequence of cones in $\Fc$ of strictly decreasing dimension such that $K$ does not intersect the relative interior of each cone. The last cone in this sequence has to be the lineality space $A$ of $\Fc$. So by this induction we get $K\subset\partial A$ which is a contradiction to fact that $\partial A=\emptyset$. Hence, there has to exist a cone $L\subset \Fc$ such that $\mathring{K}\subset \mathring{L}$.
\end{proof}

The next result is a slight variation from \cite[Lemma 6.6]{TK}. The proof is elementary and very similar to the proof given there and will be omitted.

\begin{lemma}\label{oezil}
Let $C\subset \R^n$ be a cone and $\dim C=m$. Moreover, let $D_1,\ldots, D_s\subset \R^n$ be cones such that $$C\subset \bigcup_{i=1}^s D_i,$$ where $\dim D_1=m$ and $\dim D_2,\ldots,\dim D_s<m$. Then $C\subset D_1$.
\end{lemma}

With these prerequisites we can show that for $m<n$ the fan structure of $\Wc_n^m$ is actually the coarsest possible on the set $\left|\Wc_n^m\right|$ in the sense that every other fan $\Fc\subset \R^n$ with $\left|\Fc\right|=\left|\Wc_n^m\right|$ refines $\Wc_n^m$ as a fan. In particular, this will imply that $\gT(I)$ refines $\Wc_n^m$ as a fan for a graded ideal $I\subset K[x_1,\ldots,x_n]$ with $\dim I=m$.

\begin{prop}\label{almeida}
Let $m<n$ and $\Fc\subset \R^n$ be a fan with $\left|\Fc\right|=\left|\Wc_n^m\right|$. Then for every relatively open cone $\mathring{C}$ of $\Fc$ there exists a relatively open cone $\mathring{C}_A$ of $\Wc_n^m$ such that $\mathring{C}\subset \mathring{C}_A$.
\end{prop}

\begin{proof}
For the proof note that $\Wc_n^m=\dot\bigcup_{\left|A\right|\geq n-m+1} \mathring{C}_A$ is the disjoint union of all relatively open cones of $\Wc_n$ whose defining set $A\subset\left\{1,\ldots,n\right\}$ has at least $n-m+1$ elements. By Lemma \ref{sanogo} it suffices to prove the condition for the maximal cones of $\Fc$. Let $C$ be a maximal cone of $\Fc$. Then $\dim C=m$, as $\Fc$ is pure by Proposition \ref{baumann}. Since $\dim\bigcup_{\left|A\right|> n-m+1} \mathring{C}_A=m-1<m$, there exists an $\omega\in \mathring{C}$ which is contained in the interior of some maximal cone $\mathring{C}_{A_1}$ of $\Wc_n^m$ with $\left|A_1\right|=n-m+1$. Assume there exists $v\in \mathring{C}$ such that $v\in \mathring{C}_{A_2}$ for a different maximal cone $\mathring{C}_{A_2}$ of $\Wc_n^m$. Then $\left|A_1\cap A_2\right|<n-m+1$. We have to consider two cases:
\begin{enumerate}
\item If $A_1\cap A_2\neq\emptyset$, the minimal coordinates of $\omega+v$ are attained exactly at the indices in $A_1\cap A_2$. But $\left|A_1\cap A_2\right|<n-m+1$, so $\omega+v\notin\left|\Wc_n^m\right|$. This is a contradiction to $\omega+v\in\mathring{C}\subset \left|\Wc_n^m\right|$.
\item Next we assume that $A_1\cap A_2=\emptyset$. Since $\dim C=m$, we can change the coordinates of $\omega$ which are not contained in $A_1$ independently from each other by adding or subtracting small real numbers without leaving $\mathring{C}$. The same is true for the coordinates of $v$ which are not in $A_2$. Hence, we can change every coordinate of $\omega+v$ by a small amount without leaving $\mathring{C}$, since $A_1\cap A_2=\emptyset$. But then we can assume that the minimum of the coordinates of $\omega+v$ is attained only once. Again we have $\omega+v\notin\left|\Wc_n^m\right|$ contradicting $\omega+v\in \mathring{C}\subset \left|\Wc_n^m\right|$.
\end{enumerate}
Hence, no element of $\mathring{C}$ can be contained in the relative interior of any maximal cone of $\Wc_n^m$ other than $C_{A_1}$. But then $$\mathring{C}\subset \mathring{C}_{A_1}\cup(\bigcup_{\left|A\right|>n-m+1} \mathring{C}_A).$$ Taking the topological closure this implies $C\subset C_{A_1}$ by Lemma \ref{oezil}. Since both cones have the same dimension, we also have $\mathring{C}\subset \mathring{C}_{A_1}$ by Lemma \ref{pasanen}.

\end{proof}

As a corollary Proposition \ref{almeida} implies that the generic tropical variety always refines $\Wc_n^m$ as a fan.

\begin{cor}\label{rosenberg}
Let $I\subset K[x_1,\ldots,x_n]$ be a graded ideal with $\dim I=m>0$. Then for every relatively open cone $\mathring{C}$ of $\gT(I)$ there exists a relatively open cone $\mathring{D}$ of $\Wc_n^m$, such that $\mathring{C}\subset \mathring{D}$.
\end{cor}

\begin{proof}
Since $\left|\gT(I)\right|=\left|\Wc_n^m\right|$ by \cite[Corollary 8.4]{TK} and $\Wc_n^m$ is pure $m$-dimensional, it follows from Proposition \ref{baumann} that $\gT(I)$ is a pure $m$-dimensional fan. The claim now is a consequence from Proposition \ref{almeida}.
\end{proof}

\section{Generic Tropical Varieties of Cohen-Macaulay and Almost-Cohen-Macaulay Ideals}\label{fritz}

In addition to the dimension of an ideal it is also possible to recover information on the depth of the ideal from the generic tropical variety. We will show that for an $m$-dimensional ideal $I$ with $\depth(I)>0$ the generic tropical variety is $\Wc_n^m$ as a fan if and only if $\depth(I)=\dim I$ or $\depth(I)=\dim I-1$. Thus we can read off whether $I$ is Cohen-Macaulay or almost-Cohen-Macaulay from the fan $\gT(I)$.

To define the depth of $I$ we denote the coordinate ring $K[x_1,\ldots,x_n]/I$ by $R_I$. Recall that a system of linear forms $l_1,\ldots,l_t\in R_I$ is called a \emph{ regular sequence for $R_I$} if $l_i$ is not a zero-divisor on $R_I/(l_1,\ldots,l_{i-1})$ for $i=1,\ldots,t$.

\begin{defn}
For a graded ideal $I\subset K[x_1,\ldots,x_n]$ we define the \emph{depth of $I$} to be $$\depth(I)=\max\left\{t\in \N_0: \text{ there exists a regular sequence of linear forms } l_1,\ldots,l_t\in R_I \right\}.$$
\end{defn}

The depth is bounded from above by the dimension of the ideal (see for example \cite[Proposition 1.2.12]{BRHE}). Moreover, we know that $\depth(I)\geq\depth(\gin_{\succ}(I))$ for any term order $\succ$. Equality holds if $\succ$ is a graded reverse lexicographic order. These two statements follow from \cite[Corollary 3.5 and Remark 3.6]{BRCO} together with the Auslander-Buchsbaum formula.

In general it is not possible to see the depth of $I$ in the fan $T(I)$ as the following example shows.

\begin{ex}
For $1\leq k\leq n$ consider the ideal $I=(x_1(x_1+x_2),x_2(x_1+x_2),\ldots,x_k(x_1+x_2))\subset K[x_1,\ldots,x_n]$. Then $\dim I=n-1$ and $\depth(I)=n-k$. But the tropical variety $T(I)$ always consists of only one cone $T(I)=\left\{\omega \in\R^n: \omega_1=\omega_2\right\}$ which is independent from $k$. So we have obtained a collection of ideals of every possible depth from $0$ to $n-1$ such that the tropical variety is always the same.
\end{ex}

The connection of $\depth(I)$ with $\gT(I)$ is established by the following proposition taken as a reformulation from \cite[Lemma 3.1]{HESR} and relying on \cite{ELKE}. Since a generic initial ideal $J$ is a monomial ideal, there exists a system of monomial generators of $J$. The unique smallest system of monomial generators with respect to inclusion will be called a \emph{minimal system of generators} and its elements are \emph{minimal generators} of $J$.

\begin{prop}\label{timsmail}
Let $I\subset K[x_1,\ldots,x_n]$ be a graded ideal with $\dim(I)=m$ and $\succ$ be any term order with $x_1\succ\ldots\succ x_n$. Let $\depth(\gin_{\succ}(I))=t$. Then:
\begin{enumerate}
\item Every minimal generator of $\gin_{\succ}(I)$ is divisible by one of $x_1,\ldots,x_{n-m}$.
\item $x_{n-m}^d$ is one of the minimal generators of $\gin_{\succ}(I)$ for some $d\in \N$.
\item The minimal generators of $\gin_{\succ}(I)$ are elements of $K[x_1,\ldots,x_{n-t}]$.
\item There exists a minimal generator of $\gin_{\succ}(I)$ which is divisible by $x_{n-t}$.
\end{enumerate}
In particular, if $\succ$ is the reverse lexicographic order, these statements are true for $t=\depth(I)$, since then $\depth(I)=\depth(\gin_{\succ}(I))$.
\end{prop}

Recall that by \cite[Corollary 8.4]{TK} the condition for $\omega\in\R^n$ to be in $\gT(I)$ is that the minimum of its coordinates has to be attained at least $n-m+1$ times. So Proposition \ref{timsmail} already indicates that the cases $\depth(I)=m$ and $\depth(I)=m-1$ are special. We use the following standard definition.

\begin{defn}
Let $I\subset K[x_1,\ldots,x_n]$ be a graded ideal. If $\depth(I)=\dim(I)$, then $I$ will be called \emph{Cohen-Macaulay}. If $\depth(I)=\dim(I)-1$, then $I$ is called \emph{almost-Cohen-Macaulay}.
\end{defn}

In this case the refinement $\succ_{\omega}$ of every $\omega\in \gT(I)$ with respect to an appropriate reverse lexicographic order $\succ$ yields the same generic initial ideal as with respect to $\succ$.

\begin{lemma}\label{marin}
Let $I\subset K[x_1,\ldots,x_n]$ be a graded Cohen-Macaulay or almost-Cohen-Macau\-lay ideal, $\succ$ be the degree reverse lexicographic order with $x_1\succ x_2\succ\ldots\succ x_n$ and $\omega\in \Wc_n^m\subset \R^n$ with $\omega_1=\omega_2=\ldots=\omega_{n-m+1}\leq \omega_{n-m+2},\ldots,\omega_n$. Moreover, let $\succ_{\omega}$ be the refinement of $\omega$ with respect to $\succ$. Then the reduced Gröbner bases of $g(I)$ with respect to $\succ$ and $\succ_{\omega}$ are the same for $g\in U$. In particular, $\gin_{\succ_{\omega}}(I)=\gin_{\succ}(I)$.
\end{lemma}

\begin{proof}
Since for a given degree $t$ any term containing none of $x_{n-m+2},\ldots,x_n$ is smaller than any term divisible by one of them with respect to $\succ_{\omega}$, the term orders $\succ$ and $\succ_{\omega}$ coincide up to the term $x_{n-m+1}^t$. By Proposition \ref{timsmail} the minimal generators of $\gin_{\succ}(I)$ are monomials in $K[x_1,\ldots,x_{n-m+1}]$. Then for $g\in U$ the leading terms of the reduced Gröbner basis $\Gc(g)$ of $g(I)$ are terms in $K[x_1,\ldots,x_{n-m+1}]$. Since the leading terms of two elements of $\Gc(g)$ are the same with respect to $\succ$ and $\omega_{\succ}$ every S-pair with respect to $\succ_{\omega}$ is the same as with respect to $\succ$. As $\Gc(I)$ is a Gröbner basis with respect to $\succ$, every such S-pair reduces to $0$. So the set $\Gc(g)$ is a Gröbner basis with respect to $\succ_{\omega}$ as well. Hence, $\gin_{\succ_{\omega}}(I)=\gin_{\succ}(I)$.
\end{proof}

We can now formulate the reverse statement of Proposition \ref{almeida} for Cohen-Macaulay and almost-Cohen-Macaulay ideals.

\begin{prop}\label{naldo}
Let $I\subset K[x_1,\ldots,x_n]$ be a graded Cohen-Macaulay or almost-Cohen-Macaulay ideal with $\dim I=m$. Then for every relatively open cone $\mathring{C}_A\subset \Wc_n^m$ there exists a relatively open cone $\mathring{C}$ of $\gT(I)$ with $\mathring{C}_A\subset \mathring{C}$.
\end{prop}

\begin{proof}
Let $A\subset \left\{1,\ldots,n\right\}$ with $\left|A\right|\geq n-m+1$ so $\mathring{C}_A$ is an open cone of $\Wc_n^m$. We need to show that for $\omega,\omega'\in \mathring{C}_A$ we have $\inom_{\omega}(g(I))=\inom_{\omega'}(g(I))$ for every $g\in U$. Without loss of generality we may assume $\left\{1,\ldots,n-m+1\right\}\subset A$. Let $\succ$ denote the degree reverse lexicographical order with $x_1\succ\ldots\succ x_n$. By Lemma \ref{marin} we know the reduced Gröbner basis $\Gc(g)=\left\{h_1(g),\ldots,h_s(g)\right\}$ of $g(I)$ with respect to $\succ$ is also a reduced Gröbner basis with respect to $\succ_{\omega}$ and $\succ_{\omega'}$ for $g\in U$. So $\left\{\inom_{\omega}(h_1(g)),\ldots,\inom_{\omega}(h_s(g))\right\}$ and $\left\{\inom_{\omega'}(h_1(g)),\ldots,\inom_{\omega'}(h_s(g))\right\}$ are Gröbner bases of $\inom_{\omega}(g(I))$ and $\inom_{\omega'}(g(I))$ respectively. However, all the leading terms of the $h_i(g)$ are elements of $K[x_1,\ldots,x_{n-m+1}]$. Hence, $\inom_{\omega}(h_i(g))$ and $\inom_{\omega'}(h_i(g))$ exactly consist of those terms of $h_i(g)$ which contain only variables $x_j$ for which $\omega_j$ and $\omega'_j$ respectively are minimal. But these variables are the same for $\omega$ and $\omega'$ by assumption, so we obtain $\inom_{\omega}(g(I))=\inom_{\omega'}(g(I))$. This shows that all $\omega\in \mathring{C}_A$ are contained in the same open cone $\mathring{C}$ of $T(g(I))=\gT(I)$ for $g\in U$.
\end{proof}

In the case of an $m$-dimensional Cohen-Macaulay or almost-Cohen-Macaulay ideal the generic tropical variety is equal to $\Wc_n^m$ as a fan. This generalizes the result \cite[Corollary 8.4]{TK} for this class of ideals.

\begin{cor}\label{mertesacker}
Let $I\subset K[x_1,\ldots,x_n]$ be a graded Cohen-Macaulay or almost-Cohen-Macaulay ideal with $\dim I=m$. Then $\gT(I)=\Wc_n^m$ as a fan.
\end{cor}

\begin{proof}
Let $\mathring{C}$ be a relatively open cone of $\gT(I)$. By Corollary \ref{rosenberg} there exists a cone $D$ of $\Wc_n^m$ such that $\mathring{C}\subset \mathring{D}$. On the other hand by Lemma \ref{naldo} there exists a cone $E$ of $\gT(I)$ with $\mathring{D}\subset \mathring{E}$. But then $\mathring{C}\subset \mathring{E}$ are two cones of $\gT(I)$ with $\mathring{C}\cap\mathring{E}\neq \emptyset$. This implies $\mathring{C}=\mathring{E}$ and thus $\mathring{C}=\mathring{D}$. This shows that every maximal cone of $\gT(I)$ is equal to some maximal cone of $\Wc_n^m$. By the same argument it follows that every maximal cone from $\Wc_n^m$ is equal to some maximal cone of $\gT(I)$, so the two fans are the same.
\end{proof}

To show that Corollary \ref{mertesacker} is wrong for every ideal that is not Cohen-Macaulay or almost-Cohen-Macaulay we need the following auxiliary result.

\begin{lemma}\label{proedl}
Let $c\in\N$ and $\omega\in\R^n$ be such that $0=\omega_1=\ldots=\omega_{n-m+1}$ and $c\omega_i<\omega_{i+1}$ for $i=n-m+1,\ldots,n-1$. Let $\succ$ be the reverse lexicographic order with $x_1\succ\ldots\succ x_n$. Then $\succ$ and $\succ_{\omega}$ are the same term orders for the monomials of any degree up to $c$.
\end{lemma}

\begin{proof}
Let $t\leq c$ and $x^{\nu},x^{\mu}$ be two monomials of degree $t$. We write $x^{\nu}=y_1z_1$ and $x^{\mu}=y_2z_2$, where $y_1,y_2\in K[x_1,\ldots,x_{n-m+1}]$ and $z_1,z_2\in K[x_{n-m+2},\ldots,x_n]$. If $z_1=z_2$, it is clear from the definition that $x^{\nu}\succ x^{\mu}$ if and only if $x^{\nu} \succ_{\omega} x^{\mu}$. Otherwise let $k\geq n-m+2$ be the largest index such that $\nu_k\neq \mu_k$.
Without loss of generality we may assume that no variable $x_j$ divides $x^{\nu}$ or $x^{\mu}$ for $j>k$ and that $\nu_k<\mu_k$, so $x^{\nu}\succ x^{\mu}$. For the $\omega$-weight of $x^{\nu}$ and $x^{\mu}$ we obtain the upper bound $$\wt_{\omega}(x^{\nu})\leq \wt_{\omega}(x_{k-1}^{t-\nu_k}x_k^{\nu_k})=\omega_{k-1}(t-\nu_k)+\omega_k \nu_k$$ and the lower bound $$\wt_{\omega}(x^{\mu})\geq \wt_{\omega}(x_1^{t-\mu_k} x_k^{\mu_k})=\omega_k\mu_k.$$ So it is enough to show that $\omega_{k-1}(t-\nu_k)+\omega_k\nu_k<\omega_k\mu_k$. We have

\begin{eqnarray*}
c(\omega_k\mu_k-(\omega_{k-1}(t-\nu_k)+\omega_k\nu_k)) & = & c\omega_k(\mu_k-\nu_k)-c\omega_{k-1}(t-\nu_k)\\
                                                                                                       & > & c\omega_k(\mu_k-\nu_k)-\omega_k(t-\nu_k)\\
                                                                                                       & = & \omega_k(c(\mu_k-\nu_k)-(t-\nu_k))\\
                                                                                                             & \geq & 0.
\end{eqnarray*}
The last inequality is true, since $t-\nu_k\leq c$ and $c(\mu_k-\nu_k)>c$, as we know $\mu_k>\nu_k$. It follows that $\wt_{\omega}(x^{\nu})<\wt_{\omega}(x^{\mu})$, so $x^{\nu}\succ_{\omega} x^{\mu}$. Hence, $\succ$ and $\succ_{\omega}$ coincide up to degree $c$.
\end{proof}

We can now completely characterize, when $\gT(I)$ is equal to a skeleton of the generic tropical fan for ideals of $\depth(I)>0$. If $\dim(I)=0$,  we know that $\gT(I)$ is empty, since every graded ideal of dimension $0$ contains a monomial. In the cases $\dim(I)=1$ and $\dim(I)=2$ the fan $\gT(I)$ is equal to $\Wc_n^1$ and $\Wc_n^2$ respectively by \cite[Examples 8.6, 8.7]{TK}. Note that in these cases every ideal of $\depth(I)>0$ is Cohen-Macaulay or almost Cohen-Macaulay. For ideals with arbitrary dimension $\dim(I)>0$ we have the following.

\begin{thm}\label{tosic}
Let $I\subset K[x_1,\ldots,x_n]$ be a graded ideal with $\dim(I)=m>0$ and $\depth(I)>0$. Then $I$ is Cohen-Macaulay or almost-Cohen-Macaulay if and only if $\gT(I)=\Wc_n^m$ as a fan.
\end{thm}

\begin{proof}
We show that if $\depth(I)<m-1$, then $\gT(I)\neq\Wc_n^m$ as a fan. For this let $\succ$ be the reverse lexicographical order with $x_1\succ\ldots\succ x_{n-t}\succ x_{n-t+1}\succ\ldots\succ x_n$ and $\succ'$ be the reverse lexicographical order with $x_1\succ'\ldots\succ' x_{n-t+1}\succ' x_{n-t}\succ'\ldots\succ' x_n$. Let $c$ be the maximal degree of the minimal generators of $\gin_{\succ}(I)$ and $\gin_{\succ'}(I)$. For a moment for $a,b\in\R_+$ we write $a\ll b$ if $ac<b$.

Choose $\omega,v\in\R^n$ such that $$0=\omega_1=\ldots=\omega_{n-m+1}\ll\omega_{n-m+2}\ll\ldots\ll\omega_n$$ and $$0=v_1=\ldots=v_{n-m+1}\ll v_{n-m+2}\ll\ldots\ll v_{n-t+1}\ll v_{n-t}\ll\ldots\ll v_n.$$ By Lemma \ref{proedl} we know that $\succ$ and $\succ_{\omega}$ are the same term orders up to degree $c$. Since $\gin_{\succ}(I)$ is generated by monomials of degree at most $c$, for a fixed $g\in U$ all elements of the reduced Gröbner basis $\Gc$ of $g(I)$ with respect to $\succ$ have degree at most $c$. The leading term of every element of $\Gc$ is the same with respect to $\succ$ and $\succ_{\omega}$. Thus every $S$-pair of elements of $\Gc$ reduces to zero with respect to $\succ_{\omega}$ as well. So $\Gc$ is also a Gröbner basis of $\gin_{\succ_{\omega}}(I)$. This implies $\gin_{\succ_{\omega}}(I)=\gin_{\succ}(I)$.

Again by Lemma \ref{proedl} and the same argument as before we can show that $\gin_{\succ_{v}}(I)=\gin_{\succ'}(I)$. Since $\depth(\gin_{\succ}(I))=\depth(\gin_{\succ'}(I))=t$, we know that $x_{n-t}$ divides one of the minimal generators of $\gin_{\succ}(I)$ but $x_{n-t}$ does not divide one of the minimal generators of $\gin_{\succ'}(I)$ by Proposition \ref{timsmail}. So $\gin_{\succ_{\omega}}(I)=\gin_{\succ}(I)\neq\gin_{\succ'}(I)=\gin_{\succ_v}(I)$, and it follows that $\inom_{\omega}(g(I))\neq\inom_{v}(g(I))$ for $g\in U$. Hence, $\omega$ and $v$ are not in the same relatively open cone of $\gT(I)$, but they are in the same relatively open cone of $\Wc_n^m$. This implies $\gT(I)\neq \Wc_n^m$ as a fan.

The converse of this statement, i.e. that $\depth(I)\geq m-1$ implies $\gT(I)=\Wc_n^m$ as a fan, has already been proved in Corollary \ref{mertesacker}.
\end{proof}

In particular, this theorem gives a negative answer to the question posed in the introduction of \cite{TK} whether the generic tropical variety as a fan only depends on the dimension of the ideal.

\section{Generic Tropical Varieties and Depth}

In this section we will consider a certain class of ideals $I$ with $\dim I-1>\depth(I)>0$ for which we can obtain the depth from the generic tropical varieties. These ideals have the property that of all their generic initial ideals have the same $\depth$ as $I$ itself. This makes it possible to use Proposition \ref{timsmail} on all of these. We express this property by considering the \emph{generic depth} of $I$.

\begin{defn}
For a graded ideal $I\subset K[x_1,\ldots,x_n]$ we call $$\gdepth(I)=\min\left\{\depth(\gin_{\succ}(I))\right\},$$ where the minimum is taken over all possible generic initial ideals of $I$, the \emph{generic depth of $I$}. If $\depth(I)=\gdepth(I)$, then $I$ is called a \emph{maximal-$\gdepth$} ideal.
\end{defn}

Note that since $\depth(\gin_{\succ}(I))\leq\depth(I)$ for any generic initial ideal of $I$, the ideal $I$ is a maximal-$\gdepth$ ideal if and only if $\depth(I)=\depth(\gin_{\succ}(I))$ for every generic initial ideal of $I$. We firstly describe two interesting classes of maximal-$\gdepth$ ideals.

\begin{ex}\label{harnik}
The first example is the class of strongly stable ideals. Recall that a monomial ideal $I\subset K[x_1,\ldots,x_n]$ is called strongly stable with respect to some ordering $x_{i_1}>\ldots> x_{i_n}$ of the variables $x_1,\ldots,x_n$ if for every monomial $u\in I$ we also have $x_{i_j} u x^{-1}_{i_k}\in I$ for every $x_{i_k}$ which divides $u$ and every $j<k$. Every ideal $\gin_{\succ}(I)$ is a strongly stable ideal with respect to the ordering of the variables given by $\succ$, since $\charac(K)=0$. Moreover, if $I$ is strongly stable with respect to $x_{i_1}>\ldots> x_{i_n}$ and $\succ$ is a term order with $x_{i_1}\succ\ldots\succ x_{i_n}$, then $\gin_{\succ}(I)=I$. We now explain that strongly stable ideals are maximal-$\gdepth$ ideals.

Let $I\subset K[x_1,\ldots,x_n]$ be a graded strongly stable ideal with respect to $x_1>\ldots>x_n$. Let $\succ$ be any term order with $x_{i_1}\succ\ldots\succ x_{i_n}$ for $\left\{i_1,\ldots,i_n\right\}=\left\{1,\ldots,n\right\}$. In addition, let $\succ'$ be the reverse lexicographic order with $x_{i_1}\succ'\ldots\succ' x_{i_n}$. Consider the image of $I$ under the $K$-algebra isomorphism $\phi$ that maps $x_j$ to $x_{i_j}$. Then $\phi(I)$ is a strongly stable ideal with respect to term orders with $x_{i_1}\succ\ldots\succ x_{i_n}$, so in particular with respect to $\succ$ and to $\succ'$. So we know that $\gin_{\succ}(\phi(I))=\gin_{\succ'}(\phi(I))=\phi(I)$. Let $\emptyset\neq U_1\subset \GL_n(K)$ be Zariski-open such that $\inom_{\succ}(g(I))=\gin_{\succ}(I)$ for every $g\in U_1$ and $\emptyset\neq U_2\subset \GL_n(K)$ Zariski-open such that $\inom_{\succ}(h(\phi(I)))=\gin_{\succ}(\phi(I))$ for every $h\in U_2$. Note that also the set $\emptyset\neq U_2'=\left\{h\circ \phi\in\GL_n(K): h\in U_2\right\}$ is Zariski-open, as it can be defined by the polynomials obtained by permuting the polynomials defining $U_2$ according to $\phi$. Hence, $U_1\cap U_2'\neq \emptyset$. For $k\in U_1\cap U_2'$ we have $\inom_{\succ}(k(I))=\gin_{\succ}(I)=\gin_{\succ}(\phi(I))$. In addition $\gin_{\succ'}(I)=\gin_{\succ'}(\phi(I))$ by the same argument. Hence, $\gin_{\succ}(I)=\gin_{\succ'}(I)$. Since for any reverse lexicographic order $\succ$ we have $\depth(\gin_{\succ}(I))=\depth(I)$, this implies that strongly stable ideals are maximal-$\gdepth$ ideals.
\end{ex}

\begin{ex}
The second example class is the class of ideals such that $I$ and every $\gin_{\succ}(I)$ is generated by polynomials of the same degree. We can see that these ideals are also maximal-$\gdepth$ as follows. Let $S=K[x_1,\ldots,x_n]$ with the standard $\Z$-grading and note that $S/I$ is a graded $S$-module. We denote by $\beta_{i,j}$ the graded Betti number $\beta_{i,j}=\beta_{i,j}(S/I)=\dim_K(\Tor_i^S(S/I,K))_j$. For $d\in\N$ let $S(-d)$ be the graded module $S$ with the grading given by $S(-d)_j=S_{j-d}$. Recall that for some $d\in \N$ we say that $I$ has a $d$-linear resolution if and only if $\beta_{i,i+j}=0$ for $j\neq d$. This is equivalent to the fact that the minimal graded free resolution of $S/I$ has the form $$0\rightarrow S(-d-p+1)^{\beta_p}\rightarrow\ldots\rightarrow S(-d-1)^{\beta_2}\rightarrow S(-d)^{\beta_1}\rightarrow S\rightarrow S/I\rightarrow0,$$ where $\beta_i=\beta_{i,i+d}$ is the $i$th total Betti number and $p$ is the projective dimension $p=\projdim(S/I)$ of $S/I$. Let $I\subset K[x_1,\ldots,x_n]$ be a graded ideal such that $I$ and $\gin_{\succ}(I)$ are generated by polynomials of degree $d$ for every term order $\succ$. Let $J=\gin_{\succ}(I)$ be a given generic initial ideal if $I$. We now show that $\depth(I)=\depth(J)$. As $J$ is strongly stable and generated in one degree, the minimal graded free resolution (as constructed in \cite{ELKE}) is linear. Since $\beta_{i,j}(S/I)\leq \beta_{i,j}(S/J)$ for every $i,j$ (see for example \cite[Proposition 3.3]{BRCO}), this implies that $S/I$ has a linear resolution as well. So $I$ and $J$ have linear resolutions, which in turn means that their total Betti numbers depend only on the Hilbert series of $S/I$ and $S/J$ respectively. But $H_{S/I}(t)=H_{S/J}(t)$, so the Betti numbers and in particular the projective dimensions of $S/I$ and $S/J$ are the same. By the Auslander-Buchsbaum formula $\depth(I)+\projdim(S/I)=n$ it follows that $\depth(I)=\depth(J)$.  This is true for every generic initial ideal of $I$. Hence, $I$ is a maximal-$\gdepth$ ideal.
\end{ex}

In this section we give a structural result on generic tropical varieties of maximal-$\gdepth$ ideals. We will see that these as fans are closely related to the following refinement of $\Wc_n^m$.

\begin{defn}
Let $\Wc_n^m$ be the $m$-skeleton of the standard tropical fan in $\R^n$ and $0<t<m-1$. The refinement of $\Wc_n^m$ containing all open cones $$\left\{\omega\in\R^n: \omega_{i_1}=\ldots=\omega_{i_{n-m+1}}<\omega_{i_{n-m+2}},\ldots,\omega_{i_{n-t}}<\omega_{i_{n-t+1}},\ldots,\omega_{i_n}\right\}$$ for any permutation $(i_1,\ldots,i_n)$ of $\left\{1,\ldots,n\right\}$ as maximal open cones will be called the \emph{$t$-refinement of $\Wc_n^m$} and denoted by $\Wc_n^{m,t}$.
\end{defn}

We need the following technical statement.

\begin{lemma}\label{niemeyer}
Let $I\subset K[x_1,\ldots,x_n]$ be a maximal-$\gdepth$ ideal with $\dim(I)=m<n$ and $\depth(I)=t$ for some $0<t<m-1$. Let $\omega\in \mathring{C}$ for some maximal cone $C$ of $\gT(I)$ and let $\omega_{i_1}\leq\ldots\leq\omega_{i_n}$ for $\left\{i_1,\ldots,i_n\right\}=\left\{1,\ldots,n\right\}$. Then $\omega_{i_{n-t}}<\omega_{i_{n-t+1}}$.
\end{lemma}

\begin{proof}
Without loss of generality we have $$\omega_1=\ldots=\omega_{n-m+1}<\omega_{n-m+2}\leq\ldots\leq\omega_n,$$ because $\omega\in\left|\Wc_n^m\right|$. Let us assume that $\omega_{n-t}=\omega_{n-t+1}$. For $\epsilon>0$ we define $u^{\epsilon}\in \R^n$ by

\[ u^{\epsilon}_i = \left\{
\begin{array}{ll}
\omega_i-\epsilon & \text{ for } i<n-t,\\
\omega_i & \text{ for } i=n-t,\\
\omega_i+\epsilon & \text{ for } i=n-t+1,\\
\omega_i+2\epsilon & \text{ for } i>n-t+1.
\end{array}
\right. \]

And in the same way we define $v^{\epsilon}_i\in \R^n$ by

\[ v^{\epsilon}_i = \left\{
\begin{array}{ll}
\omega_i-\epsilon & \text{ for } i<n-t,\\
\omega_i & \text{ for } i=n-t+1,\\
\omega_i+\epsilon & \text{ for } i=n-t,\\
\omega_i+2\epsilon & \text{ for } i>n-t+1.
\end{array}
\right. \]

Note that $u^{\epsilon}$ and $v^{\epsilon}$ are contained in $\gT(I)$ for every choice of $\epsilon$. Let $\succ$ be any term order. Then $x_i \succ_{u^{\epsilon}} x_{n-t}\succ_{u^{\epsilon}} x_{n-t+1} \succ_{u^{\epsilon}} x_j$ and $x_i \succ_{v^{\epsilon}} x_{n-t+1}\succ_{v^{\epsilon}} x_{n-t} \succ_{v^{\epsilon}} x_j$ for $i<n-t$, $j>n-t+1$. Since $\depth(I)=\depth(\gin_{\succ_{u^{\epsilon}}}(I))=\depth(\gin_{\succ_{v^{\epsilon}}}(I))$, by Proposition \ref{timsmail} the monomial ideal $\gin_{\succ_{u^{\epsilon}}}(I)$ contains a minimal monomial generator which is divisible by $x_{n-t}$, but none that is divisible by $x_{n-t+1}$. On the other hand $\gin_{\succ_{v^{\epsilon}}}(I)$ contains a minimal generator which is divisible by $x_{n-t+1}$, but none which is divisible by $x_{n-t}$. Hence, the reduced Gröbner bases of $\inom_{u^{\epsilon}}(g(I))$ and $\inom_{v^{\epsilon}}(g(I))$ are different with respect to the same term order $\succ$. Since the reduced Gröbner basis of an ideal is unique with respect to a given term order, this implies $\inom_{u^{\epsilon}}(g(I))\neq\inom_{v^{\epsilon}}(g(I))$, so $u^{\epsilon}$ and $v^{\epsilon}$ are in different cones of $\gT(I)$. So in every neighborhood of $\omega$ in $\gT(I)$ there are elements which are in different cones of $\gT(I)$. This is a contradiction to the fact that $\omega\in\mathring{C}$ and $C$ is maximal.
\end{proof}

We can now show that for maximal-$\gdepth$ ideals $\gT(I)$ refines $\Wc_n^{m,t}$ as a fan.

\begin{prop}\label{hunt}
Let $I\subset K[x_1,\ldots,x_n]$ be a maximal-$\gdepth$ ideal with $\dim(I)=m<n$ and $0<\depth(I)=t<m-1$. Let $C$ be a maximal cone of $\gT(I)$. Then there exists a cone $D\subset \Wc_n^{m,t}$ such that $\mathring{C}\subset \mathring{D}$.
\end{prop}

\begin{proof}
By Lemma \ref{niemeyer} we know that $\mathring{C}$ does not intersect the $m-1$-skeleton $X$ of $\Wc_n^{m,t}$. So $\mathring{C}$ must be contained in the union $\R^n\backslash X$ of the open maximal cones of $\Wc_n^{m,t}$. Since it is convex, $\mathring{C}$ is connected and thus contained in one connected component of $\R^n\backslash X$. But the connected components of $\R^n\backslash X$ are the open maximal cones itself, so $\mathring{C}$ must be contained in some maximal open cone $\mathring{D}$ of $\Wc_n^{m,t}$.
\end{proof}

This shows that the fan $\gT(I)$ is always finer than the fan $\Wc_n^{m,t}$ for maximal-$\gdepth$ ideals. We can now give a complementary result by showing that every maximal cone contains a $t$-dimensional orthant of $\R^n$. For this we will need the following basic observation from Gröbner basis theory.

\begin{lemma}\label{wiese}
Let $I\subset [x_1,\ldots,x_n]$ be a graded ideal and $\omega,\omega'\in\R^n$. Let $\succ$ be a term order and $\Gc$ be the reduced Gröbner basis of $I$ with respect to $\succ_{\omega}$. If $\inom_{\omega}(f)=\inom_{\omega'}(f)$ for every $f\in \Gc$, then $\inom_{\omega}(I)=\inom_{\omega'}(I)$.
\end{lemma}

\begin{proof}
Since $\left\{\inom_{\omega}(f):f\in\Gc\right\}$ is a reduced Gröbner basis for $\inom_{\omega}(I)$ with respect to $\succ$ (see for example \cite[Lemma 2.4.2]{MATH}), it follows that $\inom_{\omega}(I)=(\inom_{\omega'}(f):f\in\Gc)\subset\inom_{\omega'}(I)$. This implies that $\inom_{\succ_{\omega}}(I)\subset\inom_{\succ_{\omega'}}(I)$. As there cannot be a proper inclusion of two initial ideals (see \cite[Corollary 2.2.3]{MATH}), this means $\inom_{\succ_{\omega}}(I)=\inom_{\succ_{\omega'}}(I)$. Therefore we have $\inom_{\omega}(I)=\inom_{\omega'}(I)$, because $\Gc$ is also a reduced Gröbner basis of $g(I)$ with respect to $\succ_{\omega'}$.
\end{proof}

\begin{prop}\label{pizarro}
Let $I\subset K[x_1,\dots,x_n]$ be a maximal-$\gdepth$ ideal with $\dim(I)=m$, $0<\depth(I)=t<m-1$ and $c$ be the maximal total degree of a minimal generator of a generic initial ideal of $I$. Let $\omega\in \gT(I)$ with $$0=\omega_1=\ldots=\omega_{n-m+1}<\omega_{n-m+2}\leq\ldots\leq\omega_{n-t}<\omega_{n-t+1},\ldots,\omega_n$$ such that $\omega_{n-t}c<\omega_j$ for $j>n-t$ and $\omega\in \mathring{C}$ for some maximal cone $C$ of $\gT(I)$. Then $$\omega+\cone(e_{n-t+1},\ldots,e_n)\subset \mathring{C},$$ where $e_i$ denotes the $i$th standard basis vector of $\R^n$.
\end{prop}

\begin{proof}
Let $\succ$ be any term order. Then for the refinement $\succ_{\omega}$ of $\omega$ by Proposition \ref{timsmail} the generic initial ideal $\gin_{\succ_{\omega}}(I)$ is minimally generated in $K[x_1,\ldots,x_{n-t}]$ as $I$ is maximal-$\gdepth$ and $\depth(\gin_{\succ_{\omega}}(I))=\depth(I)=t$. Hence, for $g\in U$ there exists a reduced Gröbner basis $\Gc$ of $g(I)$ with respect to $\succ_{\omega}$ such that $\inom_{\succ_{\omega}}(f)\in K[x_1,\ldots,x_{n-t}]$ for every $f\in \Gc$.

We show that $\inom_{\omega'}(f)\in K[x_1,\ldots,x_{n-t}]$ for every $\omega'\in\omega+\cone(e_{n-t+1},\ldots,e_n)$ and every $f\in \Gc$. To see this we need to show that every term of $f$ which contains one of the $x_{n-t+1},\ldots,x_n$ has larger $\omega'$-weight than any term of $f$ in $K[x_1,\ldots,x_{n-t}]$. By the choice of $c$ the $\omega'$-weight of a term of $f$ in $K[x_1,\ldots,x_{n-t}]$ is bounded from above by $c\omega_{n-t}$. But any of the $x_{n-t+1},\ldots,x_n$ has weight strictly larger than $c \omega_{n-t}$. Since we already know that $f$ contains a term in $K[x_1,\ldots,x_{n-t}]$, we have $\inom_{\omega'}(f)\in K[x_1,\ldots,x_{n-t}]$. By the choice of $\omega'$ all terms in $K[x_1,\ldots,x_{n-t}]$ have the same $\omega$-weight and $\omega'$-weight. So $\inom_{\omega'}(f)=\inom_{\omega}(f)$ for $f\in\Gc$. Then by Lemma \ref{wiese} it follows that $\inom_{\omega}(g(I))=\inom_{\omega'}(g(I))$ for $g\in U$. Hence, $\omega'\in \mathring{C}$ for every $\omega'\in \omega+\cone(e_{n-t+1},\ldots,e_n)$.
\end{proof}

For maximal-$\gdepth$ ideals it is therefore possible to obtain the $\depth$ of the ideal from its generic tropical variety as shown in the following theorem.

\begin{thm}\label{vranjes}
Let $I\subset K[x_1,\ldots,x_n]$ be a maximal-$\gdepth$ ideal with $\dim(I)=m$ and $0<\depth(I)<m-1$. Then $$\depth(I)=\min\left\{ t\in \N: \gT(I)\text{ refines } \Wc_n^{m,t}\right\}.$$
\end{thm}

\begin{proof}
Let $T=\depth(I)$. By Proposition \ref{hunt} we already know that $\gT(I)$ refines $\Wc_n^{m,T}$ as a fan. On the other hand let $t<T$. Then we can choose $\omega\in\R^n$ such that $$\omega_1=\ldots=\omega_{n-m+1}<\omega_{n-m+2},\ldots,\omega_{n-t}<\omega_{n-t+1},\ldots,\omega_n$$ with $\omega\in \mathring{C}$ for some maximal cone $C$ of $\gT(I)$ and $\omega_{n-t} c<\omega_j$ for $j>n-t$, where $c$ is chosen as in Proposition \ref{pizarro}. Define $\omega'_i=\omega_i$ for $i\neq n-t$ and choose $\omega'_{n-t}>\omega_{n-t+1}$. Since $t<T$, by Proposition \ref{pizarro} we know that $\omega'\in\mathring{C}$ as well. But by definition of $\Wc_n^{m,t}$ we know that $\omega$ and $\omega'$ are in different open cones of $\Wc_n^{m,t}$. So $\gT(I)$ cannot refine $\Wc_n^{m,t}$ as a fan.
\end{proof}

\begin{rem}
Note that it is also possible to recover $\depth(I)$ from $\gT(I)$ for arbitrary graded ideals $I\subset K[x_1,\ldots,x_n]$ with $\dim I=m$ and $0<\depth(I)<m-1$ in the following way. Let $\depth(I)=t$ and $\succ$ be the reverse lexicographic term order with $x_1\succ\ldots\succ x_n$. Let $c$ be the maximal degree of a minimal generator of any generic initial ideal with respect to a reverse lexicographic term order. Choose $\omega\in \mathring{C}$ for some maximal cone $C$ of $\gT(I)$ as in Lemma \ref{proedl}. We now show that for this particular choice of $\omega$ we have $$\omega+\cone(e_{n-t+1},\ldots,e_n)\subset \mathring{C},$$ but $$\omega+\cone(e_{n-t},\ldots,e_n)\not\subset \mathring{C}.$$

Since $\succ$ and $\succ_{\omega}$ coincide up to degree $c$, this implies $\gin_{\succ_{\omega}}(I)=\gin_{\succ}(I)$. In particular, $\depth(\gin_{\succ_{\omega}}(I))=\depth(I)$. By the same proof as in Proposition \ref{pizarro} we obtain that $\omega+\cone(e_{n-t+1},\ldots,e_n)\subset \mathring{C}$.

Assume that $\omega+\cone(e_{n-t},\ldots,e_n)\subset \mathring{C}$. Let $\succ'$ be the degree reverse lexicographic term order with $x_1\succ'\ldots\succ'x_{n-t-1}\succ'x_{n-t+1}\succ'\ldots\succ'x_n\succ' x_{n-t}$. Then we define $\omega'\in\R^n$ by $\omega'_i=\omega_i$ for $i\neq n-t$ and $\omega'_{n-t}>\omega_n c$. By assumption we know that $\omega'\in\mathring{C}$. Since $\succ_{\omega'}$ and $\succ'$ coincide up to degree $c$ by Lemma \ref{proedl}, we know that $\gin_{\succ_{\omega'}}(I)=\gin_{\succ'}(I)$. As in the proof of Theorem \ref{tosic} we get that $$\gin_{\succ_{\omega}}(I)=\gin_{\succ}(I)\neq\gin_{\succ'}(I)=\gin_{\succ_{\omega'}}(I).$$ This implies $\inom_{\omega}(g(I))\neq\inom_{\omega'}(g(I))$ for $g\in U$ which is a contradiction to $\omega,\omega'\in \mathring{C}$. Hence, $\omega+\cone(e_{n-t},\ldots,e_n)\not\subset \mathring{C}$.

To obtain $\depth(I)$ from $\gT(I)$ we can therefore determine $\omega$ as described above. Then we have $$\depth(I)=\min\left\{t: \omega+\cone(e_{n-t+1},\ldots,e_n)\subset \mathring{C}\right\}$$ for this particular choice of $\omega$.
\end{rem}

As we have seen in Proposition \ref{hunt} the generic tropical variety of a maximal-$\gdepth$ ideal with $\dim I=m$ and $\depth(I)=t$ with $0<t<m-1$ always refines $\Wc_n^{m,t}$. It is also true that any of the fans $\Wc_n^{m,t}$ is the generic tropical variety of some ideal which we will see by focusing on a class of strongly stable ideals generated in degree $2$.

\begin{prop}\label{tziolis}
Let $0<m<n$ and $0<t<m-1$. The ideal $$I=(x_1,\ldots,x_{n-m-1},x_{n-m}^2,x_{n-m}x_{n-m+1},\ldots,x_{n-m}x_{n-t})\subset K[x_1,\ldots,x_n]$$ is a maximal-$\gdepth$ ideal with $\dim I=m$, $\depth(I)=t$ and $\gT(I)=\Wc_n^{m,t}$ as a fan.
\end{prop}

\begin{proof}
We first show that $\dim I=m$ and $\depth(I)=t$. Since $I$ is strongly stable with respect to the usual graded reverse lexicographic order $\succ$, we have $\gin_{\succ}(I)=I$. So $I$ has only one minimal prime by \cite[Corollary 15.25]{E} which is $(x_1,\ldots,x_{n-m})$. Thus, $\dim I=m$. To see that $\depth(I)=t$, note that $I$ is strongly stable with respect to the usual reverse lexicographic order $\succ$. Hence, $\gin_{\succ}(I)=I$ and by Proposition \ref{timsmail} it follows that $\depth(I)=n-(n-t)=t$. In particular, $I$ is a maximal-$\gdepth$ ideal (see Example \ref{harnik}).

By Proposition \ref{hunt} we know that every maximal cone $\mathring{C}$ of $\gT(I)$ is contained in some maximal cone $\mathring{D}$ of $\Wc_n^{m,t}$. So it remains to show that for every $\omega,\omega'\in\mathring{D}$ for some maximal cone $D$ of $\Wc_n^{m,t}$ we have $\inom_{\omega}(g(I))=\inom_{\omega'}(g(I))$ for $g\in U$. Let $D$ be the maximal cone of $\Wc_n^{m,t}$ given by $$\mathring{D}=\left\{\omega\in\R^n:\omega_{i_1}=\ldots=\omega_{i_{n-m+1}}<\omega_{i_{n-m+2}},\ldots,\omega_{i_{n-t}}<\omega_{i_{n-t+1}},\ldots,\omega_{i_n} \right\}$$ for some permutation $(i_1,\ldots,i_n)$ of $\left\{1,\ldots,n\right\}$. Let $\omega\in \mathring{D}$ be fixed, $\succ_{\omega}$ the refinement of $\omega$ with respect to the reverse lexicographical order $\succ$ with $x_{i_{n-m+2}}\succ\ldots\succ x_{i_n}\succ x_{i_1}\succ\ldots\succ x_{i_{n-m+1}}$ and $\Gc$ the reduced Gröbner basis of $g(I)$ with respect to $\succ_{\omega}$ for a fixed $g\in U$. Note that $x_{i_1}\succ_{\omega}\ldots\succ_{\omega} x_{i_{n-m+1}}$ and $x_{i_k}\succ_{\omega} x_{i_j}$ for $k\in \left\{n-m+2,\ldots,n-t\right\}$, $j\in\left\{n-t+1,\ldots,n\right\}$. Let $(q_1,\ldots,q_n)$ be the permutation on $\left\{1,\ldots,n\right\}$ such that $x_{q_1}\succ_{\omega}x_{q_2}\succ_{\omega}\ldots\succ_{\omega}x_{q_n}$. As in Example \ref{harnik} we know that $\gin_{\succ_{\omega}}(I)=\phi(I)$ for the $K$-algebra isomorphism $\phi$ induced by $\phi(x_j)=x_{q_j}$. So $$\gin_{\succ_{\omega}}(I)=(x_{i_1},\ldots,x_{i_{n-m-1}},x_{i_{n-m}}^2,x_{i_{n-m}}x_{i_{n-m+1}},\ldots,x_{i_{n-m}}x_{i_{n-t}}).$$ Hence, $\inom_{\succ{\omega}}(f)=x_{i_j}$ for some $j\in\left\{1,\ldots,n-m-1\right\}$ or $\inom_{\succ{\omega}}(f)=x_{i_{n-m}}x_{i_k}$ for some $k\in \left\{n-m,\ldots,n-t\right\}$ for every $f\in \Gc$. Let $\omega'\in\mathring{D}$. We now show that $\inom_{\omega}(f)=\inom_{\omega'}(f)$ for every $f\in \Gc$.

If $\inom_{\succ_{\omega}}(f)=x_{i_j}$ for some $j\in\left\{1,\ldots,n-m-1\right\}$, then by comparing weights $\inom_{\omega}(f)$ is exactly the sum of all linear terms $a_{i_k} x_{i_k}$ with $a_{i_k}\in K$, $k\in\left\{1,\ldots,n-m-1\right\}$ which appear in $f$. But the same is true for $\inom_{\omega'}(f)$, since $\omega$ and $\omega'$ have the same minimal coordinates. So in this case $\inom_{\omega}(f)=\inom_{\omega'}(f)$.

If $\inom_{\succ_{\omega}}(f)=x_{i_{n-m}}x_{i_k}$ for some $k\in\left\{n-m,\ldots,n\right\}$ we have to distinguish two subcases:
\begin{enumerate}
\item If $k=n-m$ or $k=n-m+1$, then $\inom_{\omega}(f)$ is the sum of all monomials in $K[x_{n-m},x_{n-m+1}]$ which appear in $f$. Again the same is true for $\inom_{\omega'}(f)$ by the same argument as before, so $\inom_{\omega}(f)=\inom_{\omega'}(f)$.
\item For $k>n-m+1$ we need to show that certain terms cannot appear in $f$. Firstly note that no term which is divisible by any of $x_{i_1},\ldots,x_{i_{n-m-1}}$ can appear in $f$, since $f$ is part of a reduced Gröbner basis with respect to $\succ_{\omega}$ and such a term would be divisible by a leading term of another element of $\Gc$. For the same reason $f$ cannot contain the monomial $x_{i_{n-m}} x_{i_s}$ for $s\in\left\{n-m+2,\ldots,n-t\right\}\backslash\left\{k\right\}$. Note that $x_{i_{n-m+1}}^2$ cannot appear in $f$ either, since then $\wt_{\omega}(x_{i_{n-m+1}}^2)<\wt_{\omega}(x_{i_{n-m}}x_{i_k})$. Furthermore, assume that $f$ contains the monomial $x_{i_{n-m+1}}x_{i_s}$ for some index $s\in\left\{n-m+2,\ldots,n-t\right\}\backslash\left\{k\right\}$. Then for $v\in\R^n$ with $v_{i_1}=\ldots=v_{i_{n-m+1}}<v_{i_s}<v_{i_j}$ for $j\in\left\{n-m+2,\ldots,n\right\}\backslash\left\{k\right\}$ we have $\inom_{v}(f)=x_{i_{n-m+1}}x_{i_s}$ is a monomial, since every other possible term of $f$ has greater $v$-weight. This is a contradiction to $v\in \gT(I)$. This implies $x_{i_{n-m+1}}x_{i_s}$ for $s\in\left\{n-m+2,\ldots,n-t\right\}\backslash\left\{k\right\}$ does not appear in $f$ either.

With this we can determine the initial forms $\inom_{\omega}(f)$ and $\inom_{\omega'}(f)$. As we have $\wt_{\omega}(x_{i_{n-m}}x_{i_k})=\wt_{\omega}(x_{i_{n-m+1}}x_{i_k})$, these two terms have to appear in $\inom_{\omega}(f)$, if $x_{i_{n-m+1}}x_{i_k}$ is a term of $f$. Assume there exists another term in $\inom_{\omega}(f)$, then it would have to be of the form $x_{i_{n-m}}x_{i_r}$ or $x_{i_{n-m+1}}x_{i_r}$ for some $r\in\left\{n-t+1,\ldots,n\right\}$ or of the form $x_{i_a} x_{i_b}$ for some $a,b\in\left\{n-m+2,\ldots,n-t\right\}$. The first case cannot occur, since $\wt_{\omega}(x_{i_{n-m}}x_{i_r})=\wt_{\omega}(x_{i_{n-m+1}}x_{i_r})>\wt_{\omega}(x_{i_{n-m}}x_{i_k})$. Assume that $x_{i_a} x_{i_b}$ appears in $\inom_{\omega}(f)$ for some $a,b\in\left\{n-m+2,\ldots,n-t\right\}$, then of course $\wt_{\omega}(x_{i_{n-m}} x_{i_k})=\wt_{\omega}(x_{i_a} x_{i_b})$.  But we know that $x_{i_a} x_{i_b}\succ x_{i_{n-m}} x_{i_k}$, by the choice of $\succ$. This is a contradiction to $\inom_{\succ_{\omega}}(f)=x_{i_{n-m}}x_{i_k}$, so $\inom_{\omega}(f)$ only contains the monomials $x_{i_{n-m}}x_{i_k}$ and $x_{i_{n-m+1}}x_{i_k}$.

The same is true for $\inom_{\omega'}(f)$ as we can see as follows. We show that $\inom_{\succ_{\omega'}}(f)=x_{i_{n-m}}x_{i_k}$ as well. Then by the same argument as above it follows that only the terms $x_{i_{n-m}}x_{i_k}$ and $x_{i_{n-m+1}}x_{i_k}$ appear in $\inom_{\omega'}(f)$, and thus $\inom_{\omega}(f)=\inom_{\omega'}(f)$. Since $\wt_{\omega'}(x_{i_{n-m}}x_{i_r})=\wt_{\omega'}(x_{i_{n-m+1}}x_{i_r})>\wt_{\omega'}(x_{i_{n-m}}x_{i_k})$ for $r\in\left\{n-t+1,\ldots,n\right\}$, terms of this form cannot occur as the leading term. Assume that $\inom_{\succ_{\omega'}}(f)=x_{i_a} x_{i_b}$ for some $a,b\in\left\{n-m+2,\ldots,n-t\right\}$. Then $x_{i_a} x_{i_b}\in \gin_{\succ_{\omega'}}(I)$. But we know that $\dim\gin_{\succ_{\omega'}}(I)=\dim I=m$ and $x_{i_1}\succ_{\omega'}\ldots\succ_{\omega'} x_{i_{n-m}}\succ_{\omega'} x_{i_j}$ for $j>n-m$. By Proposition \ref{timsmail} this implies that $\gin_{\succ_{\omega'}}(I)$ cannot contain a monomial which does not divide one of $x_{i_1},\ldots,x_{n-m}$ which is a contradiction to $x_{i_a} x_{i_b}\in \gin_{\succ_{\omega'}}(I)$.
\end{enumerate}
We have now shown that $\inom_{\omega}(f)=\inom_{\omega'}(f)$ for every $f\in\Gc$. Hence, by Lemma \ref{wiese} we have $\inom_{\omega}(g(I))=\inom_{\omega'}(g(I))$ for $g\in U$. Thus every maximal cone of $\Wc_n^{m,t}$ is contained in a maximal cone of $\gT(I)$. The claim now follows from this together with Proposition \ref{hunt}.
\end{proof}

This result of course raises the question if it is always true that $\gT(I)=\Wc_n^{m,t}$ for strongly stable ideal or even maximal-$\gdepth$ ideals $I\subset K[x_1,\ldots,x_n]$ with $\dim I=m$ and $0<\depth(I)=t<m-1$. Computations with gfan indicate that this is not the case. For example the ideal $I=(x_1^2,x_1x_2,x_1x_3^2,x_1x_3x_4)\subset K[x_1,\ldots,x_5]$ is strongly stable with respect to $x_1>\ldots>x_5$ and has dimension $\dim I=4$ and $\depth(I)=1$ by Proposition \ref{timsmail}. However, computing $\gT(I)$ with gfan yields that $\gT(I)$ has $60$ maximal cones. Thus $\gT(I)\neq \Wc_5^{4,1}$ which has only $30$ maximal cones.

\section{Multiplicities}

In the following we abbreviate $K[x_1,\ldots,x_n]$ by $S$ if this causes no confusion. For a finitely generated graded $S$-module $M$ we denote by $H_M(t)$ the Hilbert series of $M$. Recall that the Hilbert series of $0\neq M$ can be written as $$H_M(t)=\frac{Q_M(t)}{(1-t)^d},$$ where $Q_M(t)\in \Z[t,t^{-1}]$ is a Laurent polynomial with $Q_M(1)\neq0$ and $d$ is the Krull dimension of $M$. It is well known that $Q_M(1)\neq 0$ and this number is called the \emph{multiplicity $m(I)$ of $I$} of $M$. As always we set $m(I)=m(S/I)=Q_{S/I}(1)$ for a graded ideal $I$.

To express the multiplicity of $I$ in terms of the multiplicities of its minimal primes we use the following formula known as the associativity formula for multiplicity. Note that all minimal prime ideals of a graded ideal are graded themselves. For a minimal prime ideal $P$ of $I$ let $\ell((S/I)_P)$ denote the length of the localization of the $S$-module $S/I$ at $P$. We then have $$m(I)=\sum \ell((S/I)_P) m(P),$$ where the sum is taken over all minimal primes of $I$ such that $\dim I=\dim P$; see \cite[Formula (9.\ 4)]{VA}.

We define the multiplicity of a maximal cone in $T(I)$ in a slightly more general setting as done in \cite{DIFEST}. In \cite{DIFEST} the multiplicity of a maximal cone $C$ in $T(P)$ for a prime ideal $P$ is defined as the sum of the multiplicities of all monomial-free minimal primes of the initial ideal $\inom_C(P)$ corresponding to $C$. Note that by \cite[Theorem 1]{GR} for every minimal prime $Q$ of $\inom_C(P)$ we have $\dim Q=\dim \inom_C(P)$. For an arbitrary ideal $I\subset K[x_1,\ldots,x_n]$ this is not true and in our definition we consider only those prime ideals of $\inom_C(I)$ which have the same dimension as $\inom_C(I)$.

\begin{defn}
Let $I\subset S=K[x_1,\ldots,x_n]$ be a graded ideal and $C$ be a maximal cone of $T(I)$. Let $J=\inom_C(I)$ be the initial ideal of $I$ corresponding to $C$. Then the \emph{intrinsic multiplicity} $m(C)$ of $C$ is defined as $m(C)=\sum \ell((S/J)_P)$, where the sum is taken over all minimal primes of $J$ with $\dim P=\dim J$ which do not contain a monomial.
\end{defn}

Note that in general $T(I)$ need not be pure, so in general this definition of intrinsic multiplicities will not give rise to a tropical fan as defined in \cite[Definition 2.8]{GAKEMA}. However, we only need this definition for generic tropical varieties and these are pure by Proposition \ref{baumann}. Even if $I$ is a radical ideal and $T(I)$ a pure fan, the multiplicity of the cones of $T(I)$ need not have anything to do with the multiplicity of $I$ as the following example shows.

\begin{ex}
Let $0\leq k\leq n$ and $f_k=x_1\cdots x_k (x_1+x_2)\in K[x_1,\ldots,x_n]$. Then $m(f_k)=\deg(f_k)=k+1$. But we see that the tropical variety $T(I)=\left\{\omega\in\R^n:\omega_1=\omega_2 \right\}$ consists of only one cone. The corresponding initial ideal is $(f_k)$. By factorization this has only one monomial-free minimal prime ideal which is $(x_1+x_2)$. As $\ell((S/(f_k))_{(x_1+x_2)})=1$ the only cone of $\gT(I)$ has multiplicity $1$. So in general it is impossible to obtain the multiplicity of the ideal from the multiplicity of the maximal cones of the tropical variety, at least for ideals which are not prime.
\end{ex}

In contrast, we can now prove that generically the intrinsic multiplicities of the maximal cones in the tropical variety are constant and equal the multiplicity of the ideal. For this we first show that for a graded ideal $I$ the minimal prime ideals of the initial ideals of $I$ that correspond to the maximal cones in $\gT(I)$ contain no monomial.

\begin{prop}\label{diego}
Let $I\subset K[x_1,\ldots,x_n]$ be a graded ideal with $\dim I=m$. Let $C$ be a maximal cone of $\gT(I)$ and $\omega\in\mathring{C}$. Then no minimal prime $P$ of $\inom_{\omega}(g(I))$ with $\dim P=m$ contains a monomial for $g\in U$.
\end{prop}

\begin{proof}
Since $\gT(I)=\Wc_n^m$ as a set, we can assume $\omega_1=\ldots=\omega_{n-m+1}<\omega_j$ for $j>n-m+1$ without loss of generality. For $g\in U$ let $\inom_{\omega}(g(I))\subset P$ be a minimal prime ideal with $\dim P=m$. Assume that $P$ contains a monomial $x^{\nu}$. Since $P$ is prime, this implies that $P$ contains a variable $x_k$ for some $k$. We choose $\left\{i_1,\ldots,i_{n-m}\right\}\subset\left\{1,\ldots,n-m+1\right\}\backslash \left\{k\right\}$ and a term order $\succ$ such that $$x_{i_1}\succ x_{i_2}\succ \ldots \succ x_{i_{n-m}}\succ x_j \text{ for } j\notin \left\{i_1,\ldots,i_{n-m}\right\}.$$ We have $\gin_{\succ_{\omega}}(I)=\inom_{\succ}(\inom_{\omega}(g(I)))\subset \inom_{\succ}(P)$ with $\dim \gin_{\succ_{\omega}}(I)=\dim \inom_{\succ}(P)=m$. Let $Q$ be a minimal prime of $\inom_{\succ}(P)$. Since the dimensions coincide, $Q$ is also a minimal prime of $\gin_{\succ_{\omega}}(I)$. But $\gin_{\succ_{\omega}}(I)$ has only one minimal prime which is $(x_{i_1},\ldots,x_{i_{n-m}})$ by the choice of the term order $\succ$ (see for example \cite[Corollary 15.25]{E}). Hence, $Q$ does not contain $x_k$. This is a contradiction to the fact that $x_k\in P$ and therefore $x_k\in \inom_{\succ}(P)\subset Q$. Thus, $P$ cannot contain a monomial.
\end{proof}

\begin{rem}
Note that together with \cite[Lemma 8.2]{TK}, where $\gT(I)$ can be replaced by $T(g(I))$ for every $g\in U$, and with \cite[Corollary 3.2]{TK} this gives another, simpler proof that generic tropical varieties exist as described in \cite{TK}.
\end{rem}

To use the associativity formula for multiplicities to show that $m(C)=m(I)$ in generic tropical varieties we need to show that generically all minimal primes of $\inom_{\omega}(g(I))$ have multiplicity $1$. This we do by showing that they are linear, i.e. generated by linear forms.

We need the following two general statements and include proofs for the convenience of the reader.

\begin{lemma}\label{schmidt}
Let $P\subset K[x_1,\ldots,x_n]$ be a graded prime ideal with $\dim P=1$. Then $P$ is a linear ideal.
\end{lemma}

\begin{proof}
As $P\neq (x_1,\ldots,x_n)$, we know that $V(P)\neq\left\{0\right\}$. Let $0\neq a=(a_1,\ldots,a_n)\in V(P)\subset K^n$. Then $V(Q)=K(a_1,\ldots,a_n)$ for the linear ideal $Q=(a_ix_j-a_jx_i:i<j)$. Since $V(Q)\subset V(P)$ and both are prime, this implies $P\subset Q\subset (x_1,\ldots,x_n)$. But $\dim P=1$ and $Q\neq (x_1,\ldots,x_n)$, hence $P=Q$ is linear.
\end{proof}

\begin{lemma}\label{andersen}
For a fixed $t<n$ denote $K[x_1,\ldots,x_t]$ by $R$. Let $J\subset S=K[x_1,\ldots,x_n]$ be a graded ideal and $J\subset P\subset S$ be a minimal prime of $J$ with $\dim J=\dim P=m$. If $(J\cap R)S=J$, then also $(P\cap R)S=P$.
\end{lemma}

\begin{proof}
It is clear that $(P\cap R)S\subset P$. As $J\subset P$, we know that $J=(J\cap R)S\subset (P\cap R)S\subset P$. Since $P$ is prime, so are $P\cap R$ and $(P\cap R)S$. But $P$ is a minimal prime of $J$, hence, $(P\cap R)S=P$.
\end{proof}

With this we can prove that for $I\subset K[x_1,\ldots,x_n]$ the minimal primes of the initial ideals corresponding to the maximal cones of $\gT(I)$ of the same dimension as $I$ have multiplicity $1$.

\begin{prop}\label{artmann}
Let $I\subset S=K[x_1,\ldots,x_n]$ be a graded ideal with $\dim I=m$ and $\omega\in \mathring{C}$ for some maximal cone $C$ of $\gT(I)$. Then for every $g\in U$ every minimal prime $P$ of $\inom_{\omega}(g(I))$ with $\dim P=\dim \inom_{\omega}(g(I))$ is a linear ideal. In particular, $m(P)=1$.
\end{prop}

\begin{proof}
Without loss of generality we can assume $\omega_1=\ldots=\omega_{n-m+1}<\omega_j$ for $j>n-m+1$. Let $g\in U$ and $\inom_{\omega}(g(I))\subset P$ be a minimal prime with $\dim P=m$. Let $\Gc=\left\{f_1,\ldots,f_t\right\}$ be a reduced Gröbner basis of $g(I)$ with respect to $\succ_{\omega}$ for a term order $\succ$ with $x_1\succ \ldots\succ x_n$. Then $$(\inom_{\succ_{\omega}}(f_i):i=1,\ldots,t)=\inom_{\succ}(\inom_{\omega}(g(I)))=\gin_{\succ_{\omega}}(I).$$ Note that $x_1\succ_{\omega}\ldots\succ_{\omega} x_{n-m+1} \succ_{\omega} x_j$ for $j>n-m+1$. Let $A\subset \left\{1,\ldots,t\right\}$ be the set of all indices $i$ such that $\inom_{\succ_{\omega}}(f_i)\in K[x_1,\ldots,x_{n-m}]$. We define $\tilde{J}=(\inom_{\succ_{\omega}}(f_i): i\in A)$ to be the ideal generated by all initial forms of elements in $\Gc$ which are not divisible by $x_{n-m+1},\ldots,x_n$. Since $\tilde{J}\subset\gin_{\succ_{\omega}}(I)$, we know that $\dim \tilde{J}\geq m$. As $\gin_{\succ_{\omega}}(I)$ is a strongly stable ideal, by Proposition \ref{timsmail} there exists $1\leq k\leq t$ such that $\inom_{\succ_{\omega}}(f_k)=x_{n-m}^d$ for some $d\in\N$. Hence, $x_{n-m}^d\in \tilde{J}$. But $\tilde{J}$ is also a strongly stable ideal, so again by Proposition \ref{timsmail} it follows that $\dim \tilde{J}\leq m$. Thus, $\dim \tilde{J}=m$. We set $J=(\inom_{\omega}(f_i): i\in A)$. Then we have $$m=\dim \inom_{\omega}(g(I))\leq \dim J=\dim \inom_{\succ}(J)\leq\dim \tilde{J}=m,$$ where the first inequality holds, since $J\subset\inom_{\omega}(g(I))$, and the second one, because $\tilde{J}\subset \inom_{\succ}(J)$. So $\dim J=m$. Since $J\subset \inom_{\omega}(g(I))\subset P$ and all have the same dimension, $P$ is also a minimal prime ideal of $J$. For $i\in A$ every term of $f_i$ that has minimal $\omega$-weight has to be a term in $K[x_1,\ldots,x_{n-m+1}]$ by the choice of $\omega$. So we know that $J=(J\cap K[x_1,\ldots,x_{n-m+1}])S$. From Lemma \ref{andersen} it now follows that $P=(P\cap K[x_1,\ldots,x_{n-m+1}])S$. The ideal $\tilde{P}=P\cap K[x_1,\ldots,x_{n-m+1}]$ has dimension $m-(m-1)=1$ in $K[x_1,\ldots,x_{n-m+1}]$. By Lemma \ref{schmidt} we know that $\tilde{P}$ is linear. So $P=\tilde{P}S$ is linear as well and in particular, $m(P)=1$.
\end{proof}

\begin{thm}
Let $I\subset S=K[x_1,\ldots,x_n]$ be a graded ideal with $\dim I=m$. Then for $g\in U$ and any maximal cone $C$ of $T(g(I))$ we have $m(C)=m(I)$ is constant and equals the multiplicity of $I$.
\end{thm}

\begin{proof}
First of all note that the Hilbert series and thus the multiplicity of $I$ does not change if one passes to any initial ideal of $I$ (see for example \cite[Theorem 15.26]{E}). Moreover, the Hilbert series is of course not affected by coordinate change.

By Proposition \ref{diego} for $g\in U$ and any maximal cone $C$ of $T(g(I))=\gT(I)$ we know that every minimal prime of $\inom_{C}(g(I))$ of dimension $m$ does not contain a monomial. Moreover, by Proposition \ref{artmann} every such minimal prime of $\inom_{C}(g(I))$ has multiplicity $m(P)=1$. Thus with associativity formula for multiplicities we get $$m(C)=\sum \ell((S/\inom_{C}(g(I)))_P)=\sum \ell((S/\inom_{C}(g(I)))_P)m(P)=m(\inom_C(g(I)))=m(I),$$ as the sum is taken over all minimal primes of $\inom_C(g(I))$ of dimension $m$.
\end{proof}

\begin{rem}
The fan $\gT(I)$ equipped with the weights $m(C)$ for the maximal cones $C\in \gT(I)$ is a tropical fan in the sense of \cite[Definition 2.8]{GAKEMA}. It can be shown directly by elementary methods that the balancing condition is fulfilled for each cone of dimension $\dim I-1$. See \cite[Theorem 2.5.1]{SP} for a proof in a more general case.
\end{rem}

We briefly explain \cite[Example ${\bf (1)}$]{DIFEST} in our case. This example states that for an irreducible polynomial $f\in K[x_1,\ldots,x_n]$ the intrinsic multiplicity $m(C)$ of a given cone $C$ of $T(f)$ is exactly the lattice length of the edge corresponding to $C$ of the Newton polytope of $f$. Here, the lattice length of an edge is defined as the number of integer points on this edge minus $1$.

\begin{ex}
Let $0\neq f\in K[x_1,\ldots,x_n]$ be a homogeneous polynomial of degree $t$. Then every maximal cone of $\gT(f)$ has multiplicity $t$ as $m(g(f))=\deg g(f)=t$ for every $g\in U$. Let $N(g(f))$ be the Newton polytope of $g(f)$ for $g\in U$. By \cite[Lemma 9.1]{TK} for $g\in U$ we know that $$N(g(f))=\conv(te_1,\ldots,te_n),$$ where $e_1,\ldots,e_n$ are the standard basis vectors in $\R^n$. Now a maximal cone $C$ of $\gT(f)$ is given by $$C=\left\{\omega\in \R^n: \omega_{i_1}=\omega_{i_2}\leq\omega_{i_j} \text{ for } j\neq 1,2\right\}$$ for some coordinates $i_1,i_2$. This corresponds to the edge $\conv(te_{i_1},te_{i_2})$ of $N(g(f))$ for $g\in U$. This edge has lattice length $t$, that is $\left|\left\{\Z^n\cap\conv(te_{i_1},te_{i_2}) \right\}\right|=t+1$. So the lattice length coincides with the intrinsic multiplicity $m(C)$.
\end{ex}


\begin{thebibliography}{99}


\bibitem{BIGR}
R.\ Bieri and J.R.J.\ Groves, \emph{The geometry of the set of
characters induced by valuations}. J. Reine Angew. Math. {\bf 347},
168--195 (1984).

\bibitem{BJSST}
T.\ Bogart, A.N.\ Jensen, D.\ Speyer, B.\ Sturmfels and R.R.\
Thomas, \emph{Computing tropical varieties}. J. Symb. Comput. {\bf
42}, No. 1-2, 54--73 (2007).

\bibitem{BRCO}
W.\ Bruns and A.\ Conca, {\em Gröbner bases, initial ideals and initial algebras.}
In: L.L. Avramov et al. (Hrsg.), Homological methods in commutative algebra, IPM
Proceedings, Teheran (2004).


\bibitem{BRHE}
W.\ Bruns and J.\ Herzog, {\em Cohen-Macaulay rings Rev. Ed.}
Cambridge Studies in Advanced Mathematics {\bf 39}.
Cambridge: Cambridge University Press (1998).

\bibitem{BRGU}
W.\ Bruns and J.\ Gubeladze, {\em Polytopes, rings, and K-theory}.
Springer Monographs in Mathematics. New York, NY: Springer (2009).


\bibitem{DEST}
M.\ Develin and B.\ Sturmfels,
\emph{Tropical convexity}.
Doc. Math., J. DMV {\bf 9}, 1--27, erratum 205--206 (2004).

\bibitem{DIFEST}
A.\ Dickenstein, E.\ M.\ Feichtner and B.\ Sturmfels,
\emph{Tropical discriminants}.
Journal of the American Mathematical Society {\bf 20}, 1111--1133 (2007).

\bibitem{DR}
J.\ Draisma,
\emph{A tropical approach to secant dimensions}.
J. Pure Appl. Algebra {\bf 212}, No. 2, 349--363 (2008).

\bibitem{E}
D.\ Eisenbud, \emph{Commutative algebra. With a view toward algebraic geometry.}
Graduate Texts in Mathematics {\bf 150}, Springer (1995).

\bibitem{ELKE}
S.\ Eliahou and M.\ Kervaire, {\em Minimal resolutions of some monomial ideals.}
J. Algebra {\bf 129}, No.1, 1--25 (1990).

\bibitem{GA}
A.\ Gathmann,
\emph{Tropical algebraic geometry}.
Jahresber. Dtsch. Math.-Ver. {\bf 108}, No. 1, 3--32 (2006).

\bibitem{GAKEMA}
A.\ Gathmann, M.\ Kerber and H.\ Markwig, {\em Tropical fans and the moduli spaces of tropical curves.}
Compos. Math. {\bf 145}, No. 1, 173--195 (2009).

\bibitem{GAMA2}
A.\ Gathmann and H.\ Markwig, {\em Kontsevich's formula and the WDVV equations in tropical geometry.}
Adv. Math. {\bf 217}, No. 2, 537--560 (2008).

\bibitem{GR}
H.\ Gräbe, {\em Two remarks on independent sets.}
J. Algebr. Comb. {\bf 2}, No.2, 137--145 (1993).

\bibitem{HESR}
J.\ Herzog and H.\ Srinivasan, {\em Bounds for multiplicities.}
Trans. Am. Math. Soc. {\bf 350}, No.7, 2879--2902 (1998).

\bibitem{KAMAMA}
E.\ Katz, H.\ Markwig and T.\ Markwig,
\emph{The $j$-invariant of a plane tropical cubic}.
J. Algebra {\bf 320}, No. 10, 3832--3848 (2008).

\bibitem{ITMISH}
I.\ Itenberg, G.\ Mikhalkin and E.\ Shustin,
\emph{Tropical algebraic geometry}.
Oberwolfach Seminars {\bf 35}, Birkhäuser (2007).

\bibitem{MATH}
D.\ Maclagan, R.\ Thomas,  S.\ Faridi, L.\ Gold, A.V.\ Jayanthan, A. Khetan and T.\ Puthenpurakal,
\emph{Computational Algebra and
Combinatorics of Toric Ideals}. In: R.V.\ Gurjar (ed.) et al.,
Commutative algebra and combinatorics. Ramanujan Mathematical
Society Lecture Notes Series {\bf 4}, Ramanujan Mathematical Society (2007).

\bibitem{MI2}
G.\ Mikhalkin,
\emph{Tropical geometry and its applications}.
In: M.\ Sanz-Sol\'{e} (ed.) et al.,
Volume II: Invited lectures, European Mathematical Society (EMS), 827--852 (2006).

\bibitem{MORO}
T.\ Mora and L.\ Robbiano,
\emph{The Gr\"obner fan of an ideal}.
J. Symb. Comput. {\bf 6}, No.2--3, 183--208 (1988).

\bibitem{JE}
A.\ Nedergaard Jensen, {\em Algorithmic Aspects of Gröbner Fans and Tropical Varieties.}
PhD thesis, Aarhus (2007).

\bibitem{JEMAMA}
A.\ Nedergaard Jensen, H.\ Markwig and T.\ Markwig, \emph{An
algorithm for lifting points in a tropical variety}. Collect. Math.
{\bf 59}, No. 2, 129--165 (2008).

\bibitem{TK}
T.\ Römer and K.\ Schmitz, {\em Generic Tropical Varieties.}
arXiv:0904.0120.

\bibitem{SP}
D.\ Speyer, {\em Tropical Geometry.}
PhD Dissertation, University of California at Berkeley (2005).

\bibitem{SPST}
D.\ Speyer and B.\ Sturmfels,
\emph{The tropical Grassmannian}.
Adv. Geom. {\bf 4}, No. 3, 389--411 (2004).

\bibitem{ST}
B.\ Sturmfels, \emph{Gr\"obner bases and convex polytopes}. University
Lecture Series {\bf 8}, American Mathematical Society (1996).

\bibitem{VA}
W.\ Vasconcelos, {\em Computational Methods in Commutative Algebra and Algebraic Geometry.}
Springer, Heidelberg (1998); Paperback (corrected) printing (2004).

\end{thebibliography}
\end{document}